\documentclass[10pt]{amsart}

\newfont{\titre}{cmbx12 at 16pt}
\newtheorem{lem}{Lemma}[section]
\newtheorem{thm}{Theorem}[section]
\newtheorem{prop}{Proposition}[section]

\newtheorem{rem}{Remark}[section]

\date{}

\begin{document}

\title[Bilinear stabilization]{Feedback stabilization  for a bilinear control system under weak observability inequalities}


\author{K. Ammari}
\address{UR Analysis and Control of PDE's, UR 13E64, Department of Mathematics,
Faculty of Sciences of Monastir, University of Monastir, 5019 Monastir, Tunisia}
\email{kais.ammari@fsm.rnu.tn}
\author{M. Ouzahra}
\address{Department of  mathematics \& informatics, ENS. Univesity of Sidi Mohamed Ben Abdellah. Fes, Morocco}
\email{mohamed.ouzahra@usmba.ac.ma}

\begin{abstract}
In this paper, we discuss the feedback stabilization of  bilinear systems under  weak observation properties. In this case, the uniform stability is not guaranteed. Thus we provide  an explicit weak decay rate for all regular initial data. Applications to Schr\"odinger and wave equations are provided.
\end{abstract}


\maketitle

\tableofcontents

\section{Introduction}

Most of control systems which are used to describe processes in physics, engineering, economics.. are generally nonlinear and relatively complex, which makes the identification of mathematical models extremely   difficult. Consequently, the first investigations of different concepts in control theory were confined mainly  to  simple models, namely linear difference equations and linear ordinary differential equations (see \cite{bal2,won79,zab95}). Then there has been several work on generalizing the well known systems theory concepts to systems described by  partial differential equations, where the adequate state space is an infinite dimensional functional space (see \cite{bal81,bal82,cur78,cur95,sei01,sle89,zab95}). In recent years, bilinear systems  have been widely used in the modeling of various dynamical systems, since many real physical processes may be appropriately modeled as bilinear systems when linear models are inadequate. Also,
bilinear systems  provide a better approximation to a nonlinear system than linear ones \cite{kha00,li01,moh73,moh80}.

In this paper, we are concerned with the question of feedback stabilization of the following homogeneous bilinear system:
\begin{equation}\label{SS}
y^\prime (t) = Ay(t) + p(t)By(t),\ y(0)=y_{0},
\end{equation}
where the state space is an Hilbert  $H$ with inner product
$\langle\cdot,\cdot\rangle$ and corresponding norm $\|.\|$, the
dynamic $A$ is an unbounded operator with domain $D(A)\subset
H$ and generates a semigroup of contractions $S(t)$ on $H$, $B$ is
a linear bounded operator from $H$ into $H$ and $p(.)$ is a scalar function and represents the control.

\medskip

In order to construct a stabilizing control $p(t)$ for the system
(\ref{SS}), a natural approach is to formally compute the time rate
of change of the energy $ E(t):=\displaystyle\frac{1}{2}||y(t)||^2$,
obtaining thus
$$\displaystyle E^\prime (t) \le {\mathcal R}e
\left(p(t)\langle By(t),y(t)\rangle\right)\cdot$$
Thus, in order to make the energy nonincreasing, one may consider feedback controls
 $  p(t) = f(y(t))$ such that the following "dissipating energy inequality" holds
 $$ {\mathcal R}e \{ f(y(t))\langle By(t),y(t)\rangle \}
 < 0, \; \forall t\ge 0\cdot
$$
As a class of feedback controls that satisfy the last inequality, one can consider  the following family of controls : $$p_r(t)=-\displaystyle\frac{\langle y(t),By(t)\rangle}{\|y(t)\|^r} {\bf 1}_{\{t\ge 0 : y(t)\ne 0\}},\; r\in \mathbb{R} \cdot$$
The case  $r=0$ has been considered in many works \cite{bal,ber,bou,ouz2}. In \cite{bal}, a weak stabilization result has been established using the control $p_0(t)$ provided that the following assumption is verified:
\begin{equation}\label{obsw1}
\langle BS(t)y,S(t)y\rangle = 0, \; \forall t\ge 0 \Longrightarrow y = 0\cdot
\end{equation}
Under the assumption
\begin{equation}\label{obss1}
\exists \, T, \delta >0; \;  \; \displaystyle\int_0^T|\langle BS(t)y,S(t)y\rangle| dt \ge \delta \|y\|^2,
\;\forall \;y\in H,
\end{equation}
a strong stabilization result has been obtained using the control
$p_0(t)$, and the following estimate   (see \cite{ber,ouz2}) was given:
$$\|y(t)\| =O\left(\displaystyle\frac{1}{\sqrt t}\right)\cdot$$
  In addition, others polynomial estimates was provided in \cite{ouz-al12} using the control $p_r(t)$ with $r<2.$ The case $r=2$ has been
considered  in \cite{ouz1}, where
exponential stabilization results have been established using
$p_2(t)$ under the observation assumption  (\ref{obss1}).

Note that the inequality (\ref{obss1}) is necessary for uniform stabilization of conservative systems, so we can not expect such a degree of stability under  a weaker observability assumption. Accordingly, we will look for weak stabilization when dealing with weak observation assumptions.
 Such a question  was investigated in the  case of unbounded linear feedback control (see \cite{Kai-Nec, KM1}) and bounded nonlinear feedback (see \cite{FK} and references therein).

\medskip

In this paper, we study the weak stabilization of the bilinear system (\ref{SS}) under  weaker observation assumptions than (\ref{obss1}) using the  controls $p_r(t),\;r=0, 2.$ In the next section we first present  our stabilization results  under observation inequalities that extend the classical one (\ref{obss1}), and we provide the asymptotic estimate of the resulting state.  Then we give a stabilization result under a null-controllability like assumption.
In the third section we present  applications  to Schr\"odinger and wave  equations.

\section{Stabilization results and decay estimates}

For a fixed $\theta\in (0,1),$ we consider a couple $(K,L)$ of Banach spaces (see \cite{trieb}) such that:

\medskip

$\bullet$ $H\subset L,\; D(A)\subset K \subset H,$

$\bullet$ for all $y\in D(A); \; \|y\|_{D(A)}  \sim \|y\|_K,$

$\bullet$ the following interpolation holds: $[K,L]_\theta= H$ with;
\begin{equation}\label{interpolation}
\|y\|\le \|y\|_{L}^\theta  \|y\|_{K}^{1-\theta},\; \forall y\in D(A)\cdot
\end{equation}

\medskip

Let ${\mathcal H}: \mathbb{R}_+ \rightarrow \mathbb{R}_+$ be a continuous and increasing function on $\mathbb{R}_+$.

\subsection{Preliminary}

Let us  recall the following technical lemmas which gives useful  estimates for the problem of
stabilization.

\begin{lem} \label{lem1} (\cite{Kai-Nec}, p. 26 and \cite{KM1,KM2}).
 Let $(a_k)$ be a sequence of positive real numbers satisfying
$$
a_{k+1}\le a_k  - C a_{k+1}^{\alpha+2},\; \forall k\ge 0,
$$
where $C > 0$ and  $\alpha>-1$ are constants. Then there exists a positive constant $M$
such that:
$$
a_k \le \frac{M}{(k+1)^{\frac{1}{\alpha+1}}},\; \forall k\ge 0.
$$

\end{lem}

\begin{lem}\label{lem2}\cite{ouz-al12}.
Let $A$ generate a semigroup of contractions $S(t)$ and let $B : H \to H$ be linear and bounded.
 Then  for all $r\in ]-\infty,2]$,  the system (\ref{SS}) controlled with $p_r$ possesses a unique global mild solution, which satisfies the following decay estimate:

\begin{equation}\label{Obis}
\displaystyle\int_{0}^{T}|\langle BS(s)y(t),S(s)y(t)\rangle|ds\leq \bigg(2T^{\frac{3}{2}}\|B\|^{2}\|y_0\|^{2-r}+T^{\frac{1}{2}}\bigg)\|y(t)\|^{\frac{r}{2}}\bigg(\displaystyle\int_{t}^{t+T}
\displaystyle\frac{|\langle
y(s),By(s)\rangle|^{2}}{\|y(s)\|^{r}}ds\bigg)^{\frac{1}{2}}\cdot
\end{equation}
\end{lem}

\subsection{The case of quadratic control}

Let (\ref{SS}) be as given in the introduction. With the control $p_0(t)$, the system (\ref{SS}) becomes
\begin{equation}\label{cloosed}
y^\prime (t) = Ay(t)+F_0(y(t)),\ y(0)=y_{0},
\end{equation}
where $F_0(y)=-\langle y,By\rangle By, \; y\in H$.

\medskip

The next result discusses the asymptotic behaviour of the solution $y(t)$ of   (\ref{cloosed}).

\begin{thm} \label{stabp0}
Suppose that:
\begin{enumerate}
\item[(i)]
$A$ generates a semigroup of contractions $S(t)$,
\item[(ii)]
$B$ is linear and bounded.
\end{enumerate}
\begin{enumerate}
\item
If there exist $\delta, T>0$ such that
\begin{equation}\label{obss}
\int_{0}^{T}|\langle BS(t)z,S(t)z\rangle|dt\geq \delta \, \|z\|_L^2,\
\forall z \in D(A),
\end{equation}
then for all $y_0 \in D(A),$ the feedback law
\begin{equation}\label{p}
p_0(t)=-\langle y(t),By(t)\rangle,
\end{equation}
guarantees the  following decay estimate for the respective solution to  (\ref{SS}),
$$
\|y(t)\|^2 = O\left(t^{-\frac{\theta}{2-\theta}}\right)\ as \  t\rightarrow +\infty.
$$
\item Suppose that  the function  ${\mathcal K}: x \rightarrow x {\mathcal H}(x)$  is invertible on $\mathbb{R}_+$. If there exist $\delta, T>0$ such that
\begin{equation}\label{obsp0}
\int_{0}^{T}|\langle BS(t)z,S(t)z\rangle|dt\geq \delta \, \|z\|_K^2 \, \mathcal{H} \left(\frac{\|z\|^{2}}{\|z\|_K^2}\right),\
\forall z \in K \setminus\left\{0\right\},
\end{equation}
 then the feedback law (\ref{p})  guarantees the  following decay estimate for the respective solution to  (\ref{SS}):
\begin{equation}
\label{estimh0}
\|y(t)\|^2 = O \left(\mathcal{K}^{-1} \left(\frac{1}{ t} \right) \right) \ as \  t\rightarrow +\infty\cdot
\end{equation}
\end{enumerate}
\end{thm}

\begin{rem}
 In (\ref{obsp0}), if we take $\mathcal{H}(x) = 1, \, \forall x\ge 0 $ then $\|y(t)\|^2 = O \left(\frac{1}{t}\right),\ as \  t\rightarrow +\infty.$ So, the estimate (\ref{estimh0}) generalizes the stability result obtained in \cite{ber,ouz2} as an implication of exact observability inequality (\ref{obss1}).
\end{rem}

\begin{proof}
{\it Proof of the first assertion:}
Let $y_0\in D(A).$ We can assume that $y_0\ne0.$

According to Lemma \ref{lem2}, the system (\ref{cloosed}) possesses a unique global mild solution, which is continuous with respect to the initial state, and satisfies the following  variation of parameters formula:

\begin{equation}
    S(t)y_{0} -y(t)=\displaystyle\int_{0}^{t}\displaystyle\langle y(s),By(s)\rangle S(t-s)By(s)ds,
\end{equation}
and satisfies the  decay estimate
\begin{equation}\label{O}
\displaystyle\int_{0}^{T}|\langle BS(s)y(t),S(s)y(t)\rangle|ds\leq \bigg(2T^{\frac{3}{2}}   \|B\|  \|y_0\|^{2}+T^{\frac{1}{2}}\bigg) \bigg(\displaystyle\int_{t}^{t+T}
\displaystyle|\langle
y(s),By(s)\rangle|^{2}  ds\bigg)^{\frac{1}{2}}\cdot
\end{equation}
Moreover since $A$ is dissipative, an approximation argument (see \cite{bal2}) shows that
\begin{equation}\label{eneg}
\|y(t)\|^{2} -  \|y(s)\|^2 \leq -2\displaystyle\int_s^t|\langle y(\tau),By(\tau)\rangle|^{2} d\tau,\; 0\le s\le t\cdot
\end{equation}
It follows that
$$\|y(t)\|\leq \|y_{0}\|,\ \forall t\ge 0\cdot$$

\medskip

Let us consider the sequence
$s_{k}= \|y(kT)\|^{2}$, $k\in \mathbb{N}$. \\
Applying the inequality (\ref{eneg}) for $s=kT$ and $t=(k+1)T]$  and using (\ref{O}), we derive
\begin{equation}\label{*estimate}
s_{k}-s_{k+1}\ge  \frac{2\big ( \displaystyle\int_{0}^{T}|\langle BS(s)y(kT),S(s)y(kT)\rangle|ds\big)^2}  {\big(2T^{\frac{3}{2}}   \|B\|  \|y_0\|^{2}+T^{\frac{1}{2}}\big)^2 }
\end{equation}
which by, (\ref{interpolation}) and (\ref{obss}),  gives:
$$ \|y(kT)\|_{K}^{\frac{4(1-\theta)}{\theta}} \big ( s_{k}-s_{k+1} \big )  \geq \frac{2\delta^2 }{\big(2T^{\frac{3}{2}}   \|B\|  \|y_0\|^{2}+T^{\frac{1}{2}}\big)^2} s_{k}^{2/\theta}.$$
Moreover, we can see from nonlinear semigroup properties that for all $t\ge 0$, we have
$$\|Ay(t)\|\le \|Ay_0+F_0(y_0)\| + \|F_0(y_0)\|$$
$$
\le  \|Ay_0\|+ 2\|B\|^2 \|y_0\|^3,
$$
from which it comes (recall that $\|y_k\|_K\sim \|y\|_{D(A)}$)
\begin{equation}\label{K-D}
\|y(t)\|_K \le C(\|y_0\|) \|y_0\|_{D(A)}, \, \forall t\ge 0,
\end{equation}
with $ C(\|y_0\|)=:C_1 \big ( 1 + \|B\|^2 \|y_0\|^2  \big ), $ and  where $C_1$ is a positive constant which is independent of  $y_0\cdot$\\
Thus
$$
s_{k}-s_{k+1} \ge C'(\|y_0\|) s_{k}^{2/\theta},
$$
where $$ C'(\|y_0\|) =\frac{2\delta^2}{ \big (2T^{\frac{3}{2}}   \|B\|  \|y_0\|^{2}+T^{\frac{1}{2}} \big)^2  \big (C(\|y_0\|)  \|y_0\|_{D(A)}  \big)^{\frac{4(1-\theta)}{\theta}}}\cdot$$
Then, since $\|y(t)\|^2$ decreases in time, this implies
$$s_{k}-s_{k+1} \ge C'(\|y_0\|) s_{k+1}^{2/\theta},$$
which, by applying Lemma \ref{lem1}, gives
$$
s_k=O\left(\frac{1}{k^{\frac{\theta}{2-\theta}}}\right)\cdot
$$
Then using again that $\|y(t)\|^2$ decreases, we deduce that
$$
\|y(t)\|^2 = O\left(\frac{1}{t^{\frac{\theta}{2-\theta}}}\right)\ as \  t\rightarrow +\infty\cdot
$$

{\it Proof of the second assertion.}

Let $y_0\in D(A)\setminus \{0\}, $ let $C(\|y_0\|)$ be the constant given in (\ref{K-D}) and consider the sequence
$$e_{k}= s_k \mathcal{H}^2\left(\frac{s_k}{C(\|y_0\|)^2 \|y_0\|^2_{D(A)}}\right), \; k\in \mathbb{N}\cdot$$
Let us  observe that under the assumptions on $\mathcal{H}$, the two sequences $e_k$ and $\frac{e_k}{s_k}$ are decreasing.\\
Using (\ref{obsp0}), we derive from (\ref{*estimate})
$$ s_{k}-s_{k+1} \geq \frac{2\delta^2  }{\bigg(2T^{\frac{3}{2}} \|B\| \|y_0\|^2 + T^{\frac{1}{2}}\bigg)^2} \|y(kT\|_K^4 \mathcal{H} \big (\frac{\|y(kT\|^2}{\|y(kT\|^2_K} \big )^2\cdot $$
Moreover, it comes from (\ref{K-D}) and the increasing of  $\mathcal{H}$ that
$$
\mathcal{H} \big (\frac{\|y(kT\|^2}{\|y(kT\|^2_K} \big ) \ge \mathcal{H} \big ( \frac{s_k^2}{C(\|y_0\|)^2 \|y_0\|^2_{D(A)}}\big )\cdot
$$
Having in mind that $\|y_k\|_K\sim \|y\|_{D(A)}\ge \|y_0\|, $ we deduce that
$$ s_{k}-s_{k+1} \geq \frac{2\delta^2 C_2 }{\bigg(2T^{\frac{3}{2}} \|B\| \|y_0\|^2 + T^{\frac{1}{2}}\bigg)^2} s_k \, e_k, $$
for some  constant $C_2>0$ depending on $\theta$.
Thus, using the decreasing of $e_k$ and $\frac{e_k}{s_k}$ it comes
$$
e_k - \frac{e_k}{s_k}  s_{k+1}  \geq \frac{2\delta^2 C_2}{\bigg(2T^{\frac{3}{2}} \|B\| \|y_0\|^2+ T^{\frac{1}{2}}\bigg)^2} \, e_k^2,
$$
and
$$
e_k - e_{k+1} \geq \frac{2\delta^2 C_2}{\bigg(2T^{\frac{3}{2}} \|B\| \|y_0\|^2 + T^{\frac{1}{2}}\bigg)^2} \, e_{k+1}^2\cdot
$$
Applying Lemma \ref{lem1}, we deduce that:
$$
\|y(t)\|^2 = O \left( \mathcal{K}^{-1} \left(\frac{1}{ t}\right) \right) \ as \  t\rightarrow +\infty\cdot
$$
\end{proof}

\medskip
\subsection{The case of normalized control}

Let (\ref{SS}) be as given in the introduction and consider  the control
$$p_2(t)=-\displaystyle\frac{\langle y(t),By(t) \rangle}{\|y(t)\|^2} {\bf 1}_{\{t\ge 0 : y(t)\ne 0\}}\cdot $$
Thus the system (\ref{SS}) takes the form
\begin{equation}\label{cloosed2}
y^\prime (t) = Ay(t)+F_2(y(t)),\ y(0)=y_{0},
\end{equation}
where $F_2(y)=-\frac{\langle y,By\rangle}{\|y\|^2} By, $ if $ y\neq0$ and $F_2(0)=0.$

\medskip

In the following result we provide a uniform estimate for the solution of (\ref{cloosed2}).

\begin{thm} \label{stabp2}
Suppose that:
\begin{enumerate}
\item[(i)]
 $A$ generates a semigroup of contractions $S(t)$,
\item[(ii)] $B$ is linear and bounded.
\end{enumerate}
\begin{enumerate}
\item If the estimate (\ref{obss}) is satisfied, then  for all $y_0\in D(A)$ the feedback law
\begin{equation}\label{p2}
    p_2(t)=-\frac{\langle y(t),By(t)\rangle}{\|y(t)\|^2}  {\bf 1}_{\{t\ge 0:\; y(t)\ne 0\}},
\end{equation}
leads to the following  decay estimate  for the respective solution to  (\ref{SS})
\begin{equation}\label{estim2}
\|y(t)\|^2 \le C t^{-\frac{\theta}{1-\theta}} \|y_0\|^2_{D(A)},\; \forall t>0,\; \forall y_0\in D(A),
\end{equation}
for some constant $C>0$ which is independent of $y_0$.

\item Suppose  that    ${\mathcal H}(x)$ is invertible  on $\mathbb{R}_+$ and  that the function $x \mapsto \frac{1}{x} \, {\mathcal H}^2(x)$  is increasing on $(0,1).$ If the estimate (\ref{obsp0}) is satisfied, then the feedback law (\ref{p2})
leads to the following  decay estimate  for the respective solution to  (\ref{SS}):
\begin{equation}
\label{estimh}
\|y(t)\|^2 \leq C \, \mathcal{H}^{-1} \left(\frac{1}{ t} \right) \, \|y_0\|_{D(A)}^2, \forall t > 0, \, \forall y_0 \in D(A),
\end{equation}
for some constant $C>0$ which is independent of $y_0$.
\end{enumerate}
\end{thm}

\begin{proof}
By dissipativeness, we can assume without loss of generality that $y(t)$ does not vanish, so that the control $p_2(t)$ takes the form $p_2(t)=-\frac{\langle y(t),By(t)\rangle}{\|y(t)\|^2}.$

\medskip

{\it Proof of the first assertion.}

\medskip

Applying  Lemma \ref{lem2} for $r=2$, we deduce that system (\ref{cloosed2}) possesses a unique global mild solution   which is continuous with respect to initial states, and given by the following  variation of parameters formula
$$
S(t)y_{0} -y(t)=\displaystyle\int_{0}^{t}\displaystyle\frac{\langle y(s),By(s)\rangle}{\|y(s)\|^2} S(t-s)By(s)ds\cdot
$$
Moreover, we can show that (see \cite{bal2})
\begin{equation}\label{eneg2}
\|y(t)\|^{2} -  \|y(s)\|^2 \leq -2\displaystyle\int_s^t\frac{|\langle y(\tau),By(\tau)\rangle|^2}{\|y(\tau)\|^2} d\tau,\; 0\le s\le t,
\end{equation}
from which it comes immediately  $$\|y(t)\|\leq \|y_{0}\|,\ \forall t\ge 0\cdot$$
Moreover,   we can again see from Lemma \ref{lem2} that  the solution $y(t)$ of   (\ref{cloosed2}) satisfies the  decay estimate
\begin{equation}\label{O2}
\displaystyle\int_{0}^{T}|\langle BS(s)y(t),S(s)y(t)\rangle|ds\leq \bigg(2T^{\frac{3}{2}}\|B\|^{2}+T^{\frac{1}{2}}\bigg)\|y(t)\|\bigg(\displaystyle\int_{t}^{t+T}
\displaystyle\frac{|\langle
y(s),By(s)\rangle|^{2}}{\|y(s)\|^{2}}ds\bigg)^{\frac{1}{2}}\cdot
\end{equation}
Now, let us consider the sequence $s_{k}=\|y(kT)\|^{2}$, $k\in \mathbb{N}$.\\
 Applying the inequality (\ref{eneg2}) for $s=kT$ and $t=(k+1)T]$  and using  (\ref{obss}) and (\ref{O2}), we derive
$$
\begin{array}{lll}
  s_{k}-s_{k+1} &\ge & \frac{2\bigg ( \displaystyle\int_{0}^{T}|\langle BS(s)y(kT),S(s)y(kT)\rangle|ds\bigg )^2}  {\bigg(2T^{\frac{3}{2}}   \|B\|  +T^{\frac{1}{2}}\bigg)^2 \|y(kT)\|^2}\\
 \\
&\ge & \frac{2\delta^2}  {\bigg(2T^{\frac{3}{2}}   \|B\|  +T^{\frac{1}{2}}\bigg)^2} \frac{\|y(kT)\|^4_{L}}{ \|y(kT)\|^2},
\end{array}
$$
which by  (\ref{interpolation})   gives
$$  \|y(kT)\|_{K}^{\frac{4(1-\theta)}{\theta}} \big ( s_{k}-s_{k+1} \big )  \geq \frac{2\delta^2 }{\bigg(2T^{\frac{3}{2}}   \|B\|  +T^{\frac{1}{2}}\bigg)^2} s_{k}^{\frac{2}{\theta}  - 1}.$$
We have
$$
\begin{array}{lll}
  \|Ay(t)\| & \le &\|Ay(t)+F_2(y(t))\| + \|F_2(y(t))\| \\
  \\
  &\le& (1+2\|B\|^2)  \|y_0\|_{D(A)}\cdot
\end{array}
$$
 It follows that
\begin{equation}\label{C*}
\|y(t)\|_K \le C_* \|y_0\|_{D(A)},\; \forall t\ge 0,
\end{equation}
where $C_*$ is a positive constant which is independent of $y_0\cdot$\\
Thus
$$
s_{k}-s_{k+1}  \ge C \|y_0\|_{D(A)}^{-\frac{4(1-\theta)}{\theta}} s_{k}^{\frac{2}{\theta}  - 1},
$$
where $ C>0$ is independent of $y_0\cdot$\\
In the sequel, $C>0$ will denote a generic  constant which is independent of $y_0\cdot$\\
Since $\|y(t)\|$ decreases in time, this implies
$$s_{k}-s_{k+1}  \ge C \|y_0\|_{D(A)}^{-\frac{4(1-\theta)}{\theta}} s_{k+1}^{\frac{2}{\theta}  - 1}\cdot$$
Applying Lemma \ref{lem1}, we deduce that
$$
\begin{array}{lll}
\exists C>0,\;\; s_k &\le & \frac{C}{k^{\frac{\theta}{1-\theta}}} \left(  \|y_0\|_{D(A)}^{-\frac{4(1-\theta)}{\theta}}\right)^{\frac{\theta}{2(\theta-1)}}\\
\\
&=&\frac{C}{k^{\frac{\theta}{1-\theta}}}  \|y_0\|^2_{D(A)}\cdot
\end{array}
$$

\medskip

Then using again that $\|y(t)\|$ decreases, we deduce that
$$
\|y(t)\|^2 \le \frac{C}{t^{\frac{\theta}{1-\theta}}}   \|y_0\|^{2}_{D(A)},\;\; (C>0)\cdot
$$

\medskip
{\it Proof of the second assertion.}

We consider the sequence $$e_{k}=\mathcal{H}^2\left(\frac{s_k}{C_* \|y_0\|^2_{D(A)}}\right),\;  k\in \mathbb{N},$$
 where $C_*$ is the constant given in (\ref{C*}).

Here, the two sequences $e_k$ and $\frac{e_k}{s_k}$ decreasing.\\
We deduce from the inequalities (\ref{O2}) and (\ref{eneg2}) that
$$
s_{k}-s_{k+1}\ge  \frac{2\bigg ( \displaystyle\int_{0}^{T}|\langle BS(s)y(kT),S(s)y(kT)\rangle|ds\bigg )^2}  {\bigg(2T^{\frac{3}{2}}   \|B\|   + T^{\frac{1}{2}}\bigg)^2  \|y(kT)\|^2}
$$
which by, (\ref{obsp0}) and the increasing of $\mathcal{H}$, gives
$$ s_{k}-s_{k+1} \geq \frac{2\delta^2 C}{\bigg(2T^{\frac{3}{2}} \|B\|  + T^{\frac{1}{2}}\bigg)^2} s_k \, e_k,\; \:(C>0)\cdot$$
Then, using the decreasing of $e_k$ and $\frac{e_k}{s_k}$ it comes
$$
e_k - \frac{e_k}{s_k}  s_{k+1}  \geq \frac{2\delta^2 C }{\bigg(2T^{\frac{3}{2}} \|B\| + T^{\frac{1}{2}}\bigg)^2} \, e_k^2,
$$
and
$$
e_k - e_{k+1} \geq \frac{2\delta^2 C}{\bigg(2T^{\frac{3}{2}} \|B\|  + T^{\frac{1}{2}}\bigg)^2} \, e_{k+1}^2\cdot
$$
Applying Lemma \ref{lem1}, we deduce that
$$
\|y(t)\|^2 \leq C \, \mathcal{H}^{-1} \left(\frac{1}{t}\right) \, \|y_0\|_{D(A)}^2,\; t>0\; (C>0)\cdot
$$

\end{proof}

\subsection{Further  result }

In this part, we will establish a stability result under a null-controllability like assumption. More precisely,  we consider the following  estimate:
\begin{equation}\label{coer-par}
\int_{0}^{T}|\langle B(S(t)y_0),S(t)y_0 \rangle|dt\geq    \delta\|S(T)y_0\|_L^{2}, \; \forall y_0\in D(A),
\end{equation}
for some $\delta, T>0$.

Note that inequality (\ref{coer-par}) can be seen as an estimate is also a weak observability inequality; since  only $S(T)y_0$ may be recovered, not $y_0$.

\begin{thm} \label{stabp1-2}
Suppose that:
\begin{enumerate}
\item[(i)]
$A$ generates a semigroup of contractions $S(t)$,
\item[(ii)]
$B$ is linear and bounded,

\item[(iii)] the estimate (\ref{coer-par}) holds.
\end{enumerate}

Then for all $y_0 \in D(A),$   the respective solution to controls  $p_0(t)$ and $p_2(t)$ tends to $0$ as $t\to+\infty$.
\end{thm}

\begin{proof}

It suffices to prove the case of control $p_2(t)$ since the two cases can be treated similarly. \\
Let $y_0\in D(A)$. Then using (\ref{O2}) and (\ref{coer-par}), we obtain
\begin{equation}\label{estimate2*}
\|S(T)y(t)\|_L^2 \le C \|y_0\| \left( \int_t^{t+T} \frac{|\langle By(s)),J(y(s))\rangle|^2}{\|y(s)\|^2}  {\bf 1}_{\{s\ge 0:\, y(s)\ne 0\}} ds \right)^{1/2},
\end{equation}
for some constant  $C>0$.\\
Moreover, the estimate (\ref{eneg2}) implies the following integral convergence:
 $$\int_0^{+\infty} \frac{|\langle By(s)),J(y(s))\rangle|^2}{\|y(s)\|^2} {\bf 1}_{\{s\ge 0:\, y(s)\ne 0\}} ds  < \infty\cdot
 $$
This, together with (\ref{estimate2*}) implies that
$$S(T)y(t) \rightarrow 0 \; \mbox{in}\; L, \; \mbox{as}\; t\to +\infty\cdot$$
 Then taking into account the  inequalities (\ref{interpolation}) and (\ref{C*}), we deduce that
 $$S(T)y(t) \rightarrow 0 \; \mbox{in}\; H, \; \mbox{as}\;  t\to +\infty\cdot
 $$
Now, let $\epsilon>0$ and let $t_1>T$ be such that
$$\|S(T)y(t)\| <\epsilon,\; \forall t\ge t_1\cdot$$
Implementing this in the following variation of constants formula
 $$
 \begin{array}{lll}
   y(t+T)&= & S(T)y(t) + \int_{t}^{t+T} v(s) S(t+T-s)y(s) ds \\
   \\
    & = &S(T)y(t) + \int_{t-T}^t v(s+T) S(t-s) B y(s+T) ds, \,t\ge t_1,
 \end{array}
$$
it comes
$$
\|y(t+T)\|  \le \epsilon + \|B\|^2 \int_{t-T}^t  \|y(s+T)\| ds, ,\; \forall t\ge t_1\cdot
$$
Hence Gronwall yields
$$
\|y(t+T)\|\le e^{T  \|B\|^2 }  \epsilon,\, \forall t\ge t_1,
$$
which completes the proof.
\end{proof}

\begin{rem}
Note that if (\ref{coer-par}) holds with the norm of $H$, then one gets the convergence in $H$ for all $y_0\in H.$
\end{rem}

\section{Some applications}
Here we give some applications of Theorems \ref{stabp0} $\&$ \ref{stabp2}.

\subsection{Polynomial stabilization of bilinear coupled wave equations}

We consider the two following initial and boundary coupled problems:

\begin{equation}\label{dampedwbisbisc}
\left\{
\begin{array}{llll}
u_{tt} - \Delta u + \left(\int_\Omega a(x) \, \left|u_t \right|^2 \, dx \right) \, a(x) \,u_t + \beta\, v = 0, \, \quad &\mbox{in }& \Omega \times
(0, + \infty), \\
v_{tt} - \Delta v + \beta\, u = 0, \, \quad &\mbox{in }&  \Omega \times
(0, + \infty), \\
u = v = 0 \,, \quad &\mbox{on} \,& \partial \Omega \times (0,+ \infty),\\
u(x,0) = u_0(x), \, u_t(x,0) = u_1(x) \,,  \quad &\mbox{in }& \Omega,\\
v(x,0) = v_0(x), \, v_t(x,0) = v_1(x) \,,  \quad &\mbox{in }& \Omega,
\end{array}
\right.
\end{equation}

and

\begin{equation}\label{dampedwc}
\left\{
\begin{array}{llll}
u_{tt} - \Delta u + \frac{\left(\int_\Omega a(x) \, \left|u_t \right|^2 \, dx \right) \, a(x) \,u_t}
{\|(u(t),u_t(t),v(t),v_t(t))\|^2_{\left(H^1_0 (\Omega) \times L^2(\Omega)\right)^2}}
{\bf 1}_{E_1} + \beta \, v = 0, \, \quad &\mbox{in }& \Omega \times
(0, + \infty), \\
v_{tt} - \Delta v + \beta \, u = 0, \, \quad &\mbox{in }& \Omega \times
(0, + \infty), \\
u = v = 0 \,, \quad &\mbox{on }& \, \partial \Omega \times (0,+ \infty),\\
u(x,0) = u_0(x), \, u_t(x,0) = u_1(x) \,, \quad &\mbox{in }& \Omega, \\
v(x,0) = v_0(x), \, v_t(x,0) = v_1(x) \,,  \quad &\mbox{in }& \Omega\cdot
\end{array}
\right.
\end{equation}
where $a \in L^\infty(\Omega), a \geq 0, \beta \in \mathbb{R}^*$,
$\Omega$ is a bounded open set of $\mathbb{R}^n$ of class $\mathcal{C}^2$ and where ${\bf 1}_{E_1}$ is the characteristic function of the set $E_1:= \{t\ge 0:\; (u(t),u_t(t),v(t),v_t(t)) \ne (0,0,0,0)\}.$
\medskip

Here, we have:
$$
A = \left(
\begin{array}{llcc}
0 & I & 0 & 0 \\
\Delta  & 0 & - \beta & 0 \\
0 & 0 & 0 & I \\
- \beta & 0 & \Delta & 0
\end{array}
\right);\, D(A) = \left(\left[H^2(\Omega) \cap H^1_0 (\Omega)\right] \times H^1_0(\Omega) \right)^2 \subset H = \left[H^1_0 (\Omega) \times L^2(\Omega) \right]^2 \rightarrow H,
$$
which is  a skew-adjoint operator satisfying the assumptions $(i)-(ii)$ of Theorems \ref{stabp0} $\&$ \ref{stabp2}, and
$
B = \left(
\begin{array}{llcc}
0 & 0 & 0 & 0 \\
0  & a & 0 & 0 \\
0 & 0 & 0 & 0 \\
0 & 0 & 0 & 0
\end{array}
\right)  \in \mathcal{L}(H).$\\
Moreover, the corresponding linear equation becomes in this case:
\begin{equation}\label{conservativepbc}
\begin{cases}
\phi_{tt} -  \Delta \phi + \beta \psi = 0 \,,
\quad &\mbox{in }\, \Omega \times (0,+\infty)\,,\\
\psi_{tt} -  \Delta \psi + \beta \phi = 0 \,,
\quad &\mbox{in }\,  \Omega \times (0,+\infty)\,,\\
\phi= \psi = 0,  \quad  &\mbox{on }\; \partial \Omega \times (0,+\infty), \\
\phi(x, 0) = \phi_0(x), \phi_t(x, 0) = \phi_1(x),
\quad &\mbox{in } \; \Omega, \\
\psi(x, 0) = \psi_0(x), \psi_t(x, 0) = \psi_1(x),
\quad &\mbox{in }\;  \Omega\cdot
\end{cases}
\end{equation}
According to \cite{alabau,wehbe} we show that the observability inequality is given by
\begin{prop}
There exist $T, \beta_0 > 0$ and $c_T >0$ such that for all $0 <|\beta| <\beta_0$, the following observability inequality holds:
\begin{equation}\label{obsc}
||({\phi}_0,{\phi}_1,\psi_0,\psi_1)||^2_{\left[L^2(\Omega) \times
H^{-1}(\Omega)\right]^2} \leq \int_0^T \int_\Omega a \, |\phi_t|^2  \, dx \,dt \,, \quad
\end{equation}
for all initial data
$(\phi_0,\phi_1,\psi_0,\psi_1) \in \left[H^1_0(\Omega) \times L^2 (\Omega)\right]^2.$
\end{prop}
We remark here that we have \eqref{obss} for $L = \left[L^2(\Omega) \times
H^{-1}(\Omega)\right]^2$ and $\theta = 1/2.$

\medskip

Thus according to Theorems \ref{stabp0} $\&$ \ref{stabp2}, we have the following stabilization result for the bilinear wave equation.

\begin{thm}
We suppose that $supp \, (a)$ satisfies a Lions Geometric Control Condition (LGCC), as in \cite{wehbe}. Then, we have:
\begin{enumerate}
\item
The energy of  (\ref{dampedwc}) satisfies the estimate:
\begin{equation}\label{general-decaybisc}
E(t) := \frac{1}{2} \|(u(t),u_t(t),v(t),v_t(t)\|^2_{H} \le  \frac{C}{t} \, \|(u_0,u_1,v_0,v_1)\|_{D(A)}^2, \forall \, t > 0,
\end{equation}
and for all initial data $(u_0,u_1,v_0,v_1) \in \left(\left[H^2(\Omega) \cap H^1_0(\Omega)\right] \times H^1_0 (\Omega)\right)^2$.
\item
The energy of  (\ref{dampedwbisbisc}) satisfies the estimate:
\begin{equation}\label{general-decaybisbisc}
E(t) := \frac{1}{2} \|(u(t),u_t(t),v(t),v_t(t)\|^2_{H} = O
\left(\frac{1}{t^{1/3}}\right) \, \mbox{as}\; t \rightarrow + \infty.
\end{equation}
\end{enumerate}
\end{thm}

\subsection{Weak stabilization of bilinear wave equation}
We consider the following initial and boundary problem:
\begin{equation}\label{dampedw}
\left\{
\begin{array}{llll}
u_{tt} - \Delta u + \frac{\left(\int_\Omega a(x) \, \left|u_t \right|^2 \, dx \right) \, a(x) \,u_t}{\|(u(t),u_t(t))\|^2_{H^1_0 (\Omega) \times L^2(\Omega)}}  {\bf 1}_{\{t\ge 0:\; (u(t),u_t(t)) \ne (0,0)\}}= 0, \, \quad &\mbox{in }&   \Omega \times
(0, + \infty), \\
u = 0 \,, \quad &\mbox{on }& \, \partial \Omega \times (0,+ \infty),\\
u(x,0) = u_0(x), \, u_t(x,0) = u_1(x) \,,  \quad &\mbox{in }& \Omega,
\end{array}
\right.
\end{equation}
and
\begin{equation}\label{dampedwbisbis}
\left\{
\begin{array}{llll}
u_{tt} - \Delta u + \left(\int_\Omega a(x) \, \left|u_t \right|^2 \, dx \right) \, a(x) \,u_t = 0, \, \quad &\mbox{in }& \Omega \times
(0, + \infty), \\
u = 0 \,, \quad &\mbox{on }&\, \partial \Omega \times (0,+ \infty),\\
u(x,0) = u_0(x), \, u_t(x,0) = u_1(x) \,,  \quad &\mbox{in }& \Omega,
\end{array}
\right.
\end{equation}
where $a \in L^\infty(\Omega), a \geq 0$ and
$\Omega$ is a bounded open set of $\mathbb{R}^n$ of class $\mathcal{C}^2$.

\medskip

In this case, we have:
$$
A = \left(
\begin{array}{llcc}
0 & I \\
\Delta  & 0
\end{array}
\right) : D(A) = \left[H^2(\Omega) \cap H^1_0 (\Omega)\right] \times H^1_0(\Omega) \subset H = H^1_0 (\Omega) \times L^2(\Omega) \rightarrow H^1_0 (\Omega) \times L^2(\Omega),
$$
and $A$ is a skew-adjoint operator satisfying $(i)-(ii)$ of Theorems \ref{stabp0}  and \ref{stabp2}, and we have
$$B = \left(
\begin{array}{llcc}
0 & 0 \\
0  & a
\end{array}
\right)  \in \mathcal{L}(H).$$

The corresponding linear equation is:

\begin{equation}\label{conservativepb}
\left\{
\begin{array}{llll}
\phi_{tt} -  \Delta \phi = 0 \,,
\quad &\mbox{in }&  \Omega \times (0,+\infty)\,,\\
\phi=0,  \quad &\mbox{in }&  \partial \Omega \times (0,+\infty), \\
\phi(x, 0) = \phi_0(x), \phi_t(x, 0) = \phi_1(x),
\quad  &\mbox{in }& \Omega.
\end{array}
\right.
\end{equation}

According to \cite{math,phunghdr} we show that the observability inequality is given by
\begin{prop}
There exists $T$ and $c_T >0$ such that the following observability inequality holds:
\begin{eqnarray}\label{obs}
&&||({\phi}_0,{\phi}_1)||^2_{\left[H^2(\Omega) \cap H^1_0(\Omega) \right] \times H^1_0(\Omega)} \, \exp {\left[ - c_T \,
\frac{||({\phi}_0,{\phi}_1)||_{\left[H^2(\Omega) \cap H^1_0(\Omega) \right] \times H^1_0 (\Omega)}}{||({\phi}_0,{\phi}_1)||_{H^1_0(\Omega)\times L^2 (\Omega)}} \right]} \nonumber \\
&&  \hspace{1cm}\leq\int_0^T \int_\Omega a \, |\phi_t|^2  \, dx \,dt \,, \quad
\end{eqnarray}
for all non-identically zero initial data
$(\phi_0,\phi_1) \in \left[H^2(\Omega) \cap H^1_0(\Omega) \right] \times H^1_0 (\Omega).$
\end{prop}
We remark here that we have \eqref{obsp0} for ${\mathcal H}(x) = \exp(- \frac{c_T}{x^{1/2}}), \, \forall \, x > 0.$

\medskip

Thus according to Theorems \ref{stabp0}  and \ref{stabp2} we have the following stabilization result for the bilinear wave equation.

\begin{thm}
We suppose that $meas(supp \, (a)) \neq 0$. We have:
\begin{enumerate}
\item
The energy of  (\ref{dampedw}) satisfies the estimate:
\begin{equation}\label{general-decaybis}
E(t) := \frac{1}{2} \|(u(t),u_t(t)\|^2_{H} \le  \frac{C}{(\ln (1 + t))^{2}} \, \|(u_0,u_1)\|_{D(A)}^2, \forall \, t > 0,
\end{equation}
and for all initial data $(u_0,u_1) \in \left[H^2(\Omega) \cap H^1_0(\Omega)\right] \times H^1_0 (\Omega)$.
\item For all initial data $(u_0,u_1) \in \left[H^2(\Omega) \cap H^1_0(\Omega)\right] \times H^1_0 (\Omega), $ the energy of  (\ref{dampedwbisbis}) satisfies the estimate:
\begin{equation}\label{general-decaybisbis}
E(t) := \frac{1}{2} \|(u(t),u_t(t)\|^2_{H} = O
\left(\frac{1}{(\ln (t))^{2}}\right), \, t \rightarrow + \infty\cdot
\end{equation}
\end{enumerate}
\end{thm}

\subsection{Weak stabilization of bilinear Schr\"odinger equation}
We consider the following initial and boundary problem:
\begin{equation}\label{dampedsh}
\left\{
\begin{array}{llll}
u_{t} - i \Delta u + \frac{\left(\int_\Omega a(x) \, \left|u \right|^2 \, dx \right) \, a(x) \,u}{\|u(t)\|^2_{L^2(\Omega)}}  {\bf 1}_{\{t\ge 0:\; u(t) \ne 0\}}= 0, \, \quad &\mbox{in }& \Omega \times
(0, + \infty), \\
u = 0 \,, \quad &\mbox{on }&\, \partial \Omega \times (0,+ \infty),\\
u(x,0) = u_0(x) \,,  \quad &\mbox{in }& \Omega\cdot
\end{array}
\right.
\end{equation}
and
\begin{equation}\label{dampedshbis}
\left\{
\begin{array}{llll}
u_{t} - i \Delta u + \left(\int_\Omega a(x) \, \left|u \right|^2 \, dx \right) \, a(x) \,u \, = 0, \, \quad &\mbox{in }& \Omega \times
(0, + \infty), \\
u = 0 \,, \quad &\mbox{in }&\, \partial \Omega \times (0,+ \infty),\\
u(x,0) = u_0(x) \,,  \quad &\mbox{in }& \Omega,
\end{array}
\right.
\end{equation}
where $a \in L^\infty(\Omega), a \geq 0$ and
$\Omega$ is a bounded open set of $\mathbb{R}^n$ of class $\mathcal{C}^2$.

\medskip

In this case, we have:
$$
A = i \Delta : D(A) = H^2(\Omega) \cap H^1_0 (\Omega) \subset H = L^2(\Omega) \rightarrow  L^2(\Omega),
$$
and $A$ is a skew-adjoint operator satisfying assumptions $(i)-(ii)$ of Theorems \ref{stabp0} and \ref{stabp2}, and we have
$B = a \, I\!d \in \mathcal{L}(H).$

\medskip

Moreover the linear equation becomes in this case:

\begin{equation}\label{conservativepbsh}
\left\{
\begin{array}{llll}
\phi_{t} -  i \Delta \phi = 0 \,,
\quad &\mbox{in }& \Omega \times (0,+\infty)\,,\\
\phi=0,  \quad & \mbox{on }&  \partial \Omega \times (0,+\infty), \\
\phi(x, 0) = \phi_0(x),  \quad & \mbox{in }& \Omega.
\end{array}
\right.
\end{equation}
From  \cite{math,phunghdr},  we can show that the following observability inequality:
\begin{prop}
There exists $T$ and $c_T >0$ such that the following observability inequality holds:
\begin{eqnarray}\label{obsshw}
&&||\phi_0||^2_{H^2(\Omega) \cap H^1_0(\Omega)} \, \exp {\left[ - c_T \,
\frac{||\phi_0||_{H^2(\Omega) \cap H^1_0(\Omega)}}{||{\phi}_0||_{L^2 (\Omega)}} \right]} \nonumber \\
&&  \hspace{1cm}\leq\int_0^T \int_\Omega a \, |\phi|^2  \, dx \,dt \,, \quad
\end{eqnarray}
for all non-identically zero initial data
$\phi_0 \in H^2(\Omega) \cap H^1_0(\Omega).$
\end{prop}
Here,  \eqref{obsp0} holds for ${\mathcal H}(x) = \exp(- \frac{c_T}{x^{1/2}}), \, \forall \, x > 0.$
We also notice  that the constant $c_T$ in (\ref{obsshw}) can be taken large enough, so that the function $\frac{1}{x} {\mathcal H}^2(x)$ is increasing on $(0,1)$.
\medskip

By applying   Theorems \ref{stabp0} and \ref{stabp2}, we obtain the following stabilization result for the bilinear Schr\"odinger equation.

\begin{thm}
We suppose that $meas(supp \, (a)) \neq 0$. We have:
\begin{enumerate}
\item
The energy of  (\ref{dampedsh}) satisfies the estimate:
\begin{equation}\label{general-decaybissh}
E(t) := \frac{1}{2} \|u(t)\|^2_{H} \le  \frac{C}{(\ln ( 1 + t))^{2}} \, \|u_0\|_{D(A)}^2, \forall \, t > 0,
\end{equation}
and for all initial data $u_0 \in H^2(\Omega) \cap H^1_0(\Omega)$.
\item
The energy of  (\ref{dampedshbis}) satisfies the estimate:
\begin{equation}\label{general-decaybisshbis}
E(t) := \frac{1}{2} \|u(t)\|^2_{H} = O \left(\frac{1}{(\ln (t))^{2}}\right), \,  t \rightarrow + \infty.
\end{equation}
\end{enumerate}
\end{thm}

\end{document}